\documentclass [11pt]{amsart}
\usepackage[english]{babel}
\usepackage{amsfonts}
\usepackage{amssymb}
\usepackage{amsmath}

\newtheorem{thm}{Theorem}[section]
\newtheorem{prop}[thm]{Proposition}

\newtheorem{lem}[thm]{Lemma}

\newcommand{\de}{\partial}

\newcommand{\C}{\mathbb{C}}
\newcommand{\R}{\mathbb{R}}

\newcommand{\Op}[1]{\mathcal{O}_{#1}}

\newcommand{\Z}{\mathbb{Z}}
\newcommand{\Q}{\mathbb{Q}}

\newcommand{\ti}[1]{\tilde{#1}}

\newcommand{\db}{\overline{\partial}}
\newcommand{\Ric}{\mathrm{Ric}}

\newcommand{\ov}[1]{\overline{#1}}

\newcommand{\tr}{\mathrm{tr}}

\newcommand{\spn}{\mathrm{Span}}

\numberwithin{equation}{section}

\begin{document}

\title[Uniqueness of $\mathbb{CP}^n$]{Uniqueness of $\mathbb{CP}^n$}
\author{Valentino Tosatti}
\begin{abstract}
We give an exposition of a theorem of Hirzebruch, Kodaira and Yau which proves the uniqueness
of the K\"ahler structure of complex projective space, and of Yau's resolution of the Severi Conjecture.
\end{abstract}

\address{Department of Mathematics, Northwestern University, 2033 Sheridan Road, Evanston, IL 60208}
\email{tosatti@math.northwestern.edu}
\thanks{Supported in part by a Sloan Research Fellowship and NSF grant DMS-1308988. I am grateful to Yuguang Zhang and to the referee for helpful comments.}

\maketitle

\section{Introduction}
It is a classical result in complex analysis that every simply connected closed Riemann surface is biholomorphic to the projective line $\mathbb{CP}^1$. The purpose of this note is to explain in detail two higher-dimensional generalizations of this fact.

\begin{thm}[Hirzebruch, Kodaira \cite{hk}, Yau \cite{yaupnas}] \label{main1}
If a K\"ahler manifold $M$ is homeomorphic to $\mathbb{CP}^n$ then $M$ is biholomorphic to it.
\end{thm}

More precisely, Hirzebruch and Kodaira proved this for all $n$ odd, leaving open the case of $n$ even which was finally solved by Yau.
Also, Hirzebruch and Kodaira assumed that $M$ is diffeomorphic to $\mathbb{CP}^n$, and this was relaxed to homeomorphic after work of Novikov.
When $n=2$, a stronger result holds, which was known as the Severi Conjecture \cite{Se}, and was solved by Yau:

\begin{thm}[Yau \cite{yaupnas}]\label{main2} If a compact complex surface $M$ is homotopy equivalent to
$\mathbb{CP}^2$ then it is biholomorphic to it.
\end{thm}

A brief outline of the proofs of these theorems is the following. From the assumptions, using the Hirzebruch-Riemann-Roch theorem, one deduces that either $M$ is Fano (i.e. $c_1(M)$ can be represented by a K\"ahler metric) or else the canonical bundle $K_M$ is positive (i.e. $-c_1(M)$ can be represented by a K\"ahler metric). The second case can only arise when $n$ is even. When $M$ is Fano a geometric argument shows that $M$ is biholomorphic to $\mathbb{CP}^n$, which settles the case when $n$ is odd. On the other hand, when $K_M$ is positive then a key inequality between Chern numbers holds, as shown by Yau. Furthermore, in our case we have that equality holds, and this implies that $M$ is biholomorphic to the unit ball in $\mathbb{C}^n$, which is absurd because $M$ is compact.

The details are presented in Section \ref{sectproof}, mostly following the original sources (together with a small simplification of part of the argument from \cite{ko}), and in Section \ref{sectclose} we discuss a natural conjectural extension of these theorems, and how it is related to another well-known open problem.

\section{Proofs of the main results}\label{sectproof}
\begin{proof}[Proof of Theorem \ref{main1}] The fact that $M$ is K\"ahler gives us
the Hodge decomposition on cohomology, which we will use repeatedly. From the hypothesis we see that
$$H^2(M,\Z)\cong \Z,\quad H^1(M,\C)\cong 0\cong H^{0,1}(M),$$
$$H^2(M,\C)\cong\C\cong H^{2,0}(M)\oplus H^{1,1}(M)\oplus H^{0,2}(M),$$
and since $H^{2,0}(M)\cong H^{0,2}(M)$, we see that they are both zero,
while $H^{1,1}(M)\cong\C$. Thanks to the vanishing of $H^{0,1}(M)$ and $H^{0,2}(M)$, the exponential exact sequence gives
that the first Chern class map
$$c_1:\textrm{Pic}(M)\to H^2(M,\Z)\cong\Z,$$
is an isomorphism, where as usual the Picard group $\textrm{Pic}(M)$ is the group of isomorphism classes of holomorphic line bundles on $M$.

\begin{lem}\label{lemma1}
$M$ is projective and its holomorphic Euler characteristic satisfies
$$\chi(M,\mathcal{O}):=\sum_{p=0}^n (-1)^p\dim H^{0,p}(M)=1.$$ 
\end{lem}
\begin{proof}
Choose a K\"ahler form $\ti{\omega}$ on $M$. Its cohomology class
$[\ti{\omega}]$ lies in $H^2(M,\R)\cong \R$ so we can rescale $\ti{\omega}$ to
get another K\"ahler form $\omega$ whose cohomology class generates
$H^2(M,\Z)\cong\Z$. We have that $\int_M \omega^n>0$ because this equals $n!$ times the total volume of $M$ measured using the K\"ahler metric $\omega$.
On $\mathbb{CP}^n$ a generator $\alpha$ of $H^2(M,\Z)$
satisfies $\langle \alpha^{\smile n},[\mathbb{CP}^n]\rangle=\pm 1$, and since $\omega$ is K\"ahler we have that $\int_M \omega^n=1$. Since $c_1$ is an isomorphism, there exists $L\to M$ a holomorphic
line bundle whose first Chern class is $[\omega].$
If $h$ is a smooth Hermitian metric on the fibers of $L$ then its curvature form $\gamma$ is a closed real $(1,1)$ form cohomologous to $c_1(L)=[\omega]$. By the $\de\db$-Lemma, which holds because $M$ is K\"ahler, there is a smooth real-valued function $\psi$ on $M$ such that $\omega=\gamma+\sqrt{-1}\de\db\psi$. The Hermitian metric $\ti{h}=e^{-\psi}h$ on $L$ then has curvature form equal to $\omega$, and so $L$ is a positive line bundle. Thanks to the Kodaira Embedding Theorem \cite[Proposition 5.3.1]{Hu}, $L$ is ample and the manifold $M$ is projective.

Since $\int_M\omega^n\neq 0$, it follows that the classes $[\omega^k]\in H^{k,k}(M)$ are nonzero for $1\leq k\leq n$,
and as above the Hodge decomposition implies that $H^{p,q}(M)=0$ if $p\neq q$.
This gives that the holomorphic Euler characteristic of $M$ satisfies
$\chi(M,\mathcal{O})=1.$ 
\end{proof}

Recall the following definition: if $F\to M$ is a real vector bundle, then its Pontrjagin classes are defined to be
$p_i(F)=(-1)^i c_{2i}(F\otimes\C)\in H^{4i}(M,\Z)$, where $c_{2i}$ denotes the $(2i)^{th}$ Chern class of the complex vector bundle $F\otimes\C$.
 If $F=TM$ we just write $p_i(M)$.
Now we need the following theorem, which we will quote without proof.
\begin{thm}[Novikov \cite{nov}] The rational Pontrjagin classes of a closed smooth manifold are invariant under homeomorphism.
\end{thm}
Here the rational Pontrjagin classes are just the images of $p_i(M)$ under the natural map $H^{4i}(M,\Z)\to H^{4i}(M,\Q)$.
Since our manifold $M$ has torsion-free integral cohomology, we obtain in our case the invariance of the integral Pontrjagin classes.
In particular if $f:M\to\mathbb{CP}^n$ is the given homeomorphism, then $f^* p_i(\mathbb{CP}^n)=p_i(M)$ for all $i$.
Notice that if $f$ is assumed to be a diffeomorphism then this is obvious since $f^*(T\mathbb{CP}^n)\cong TM$ is an isomorphism of real vector bundles, which induces
an isomorphism of complex vector bundles $f^*(T\mathbb{CP}^n\otimes \C)\cong TM\otimes\C$ which therefore preserves the Chern classes, so we do not need
Novikov's theorem in that case. On the other hand, it is in general false that $f^*c_i(\mathbb{CP}^n)\cong c_i(M)$ when $f$ is a diffeomorphism, which is why we are forced to work with Pontrjagin classes instead of Chern classes.

\begin{lem}\label{lemma2}
The holomorphic Euler characteristic of $M$ satisfies
\begin{equation}\label{key}
\chi(M,\mathcal{O})=\int_M e^{\frac{c_1(M)}{2}}
\left(\frac{\omega/2}{\sinh (\omega/2)}\right)^{n+1}.
\end{equation}
\end{lem}
\begin{proof}
If $H$ denotes the
hyperplane class on $\mathbb{CP}^n$ then it is well-known (see e.g. \cite[Example 15.6]{ms}) that
$$p_i(\mathbb{CP}^n)=\binom{n+1}{i}H^{2i},$$
for $0\leq i\leq \lfloor{\frac{n}{2}} \rfloor$.
Moreover the fact that $f$ is a homeomorphism implies that $f^*H$ is a generator of $H^2(M,\Z)$ and
so $f^*H=\pm [\omega]$. Putting these together we get
\begin{equation}\label{pont}
p_i(M)=\binom{n+1}{i}[\omega^{2i}],
\end{equation}
for $0\leq i\leq \lfloor{\frac{n}{2}} \rfloor$.
The Hirzebruch-Riemann-Roch Theorem \cite[Theorem 20.3.2]{Hi} says that for any holomorphic line bundle $F$ on $M$ we have
$$\chi(M,F):=\sum_{p\geq 0}(-1)^p\dim H^p(M,F)=\int_M e^{c_1(F)} {\rm Td}(M),$$
where ${\rm Td}(M)$ is the Todd genus of $M$. This is defined in terms of the Chern classes of $M$, but since in our case we only know the Pontrjagin classes of $M$, we need to express ${\rm Td}(M)$ as much as possible in terms of these. To do this, we use the identity \cite[p.150, (6*)]{Hi}
$${\rm Td}(M)=e^{\frac{c_1(M)}{2}}\hat{A}(M),$$
where the $\hat{A}$ genus of $M$ is defined as follows (see \cite{Hi} for details). We formally write
$$\sum_{j\geq 0}p_j(M)x^j=\prod_{j\geq 1}(1+\gamma_j x),$$
for some symbols $\gamma_j$, and let
$$\hat{A}(M)=\prod_{j\geq 0}\frac{\sqrt{\gamma_j}/2}{\sinh (\sqrt{\gamma_j}/2)},$$
which is therefore a polynomial in the Pontrjagin classes $p_j(M)$.
Taking $F=\mathcal{O}$ in the Hirzebruch-Riemann-Roch formula (where $\mathcal{O}$ is the trivial line bundle) gives
$$\chi(M,\mathcal{O})=\int_M e^{\frac{c_1(M)}{2}}\hat{A}(M).$$
Now thanks to \eqref{pont} we have
$$\sum_{j\geq 0}p_j(M)x^j=(1+[\omega^2] x)^{n+1},$$
which gives $\gamma_1=\dots=\gamma_{n+1}=[\omega^2]$ and $\gamma_j=0$ for $j>n+1$. Thus, we obtain the key identity \eqref{key}.
\end{proof}

In order to proceed with the proof, we need to determine $c_1(M)$.

\begin{lem}\label{lemma3}
We have that $c_1(M)$ equals either $(n+1)[\omega]$ or $-(n+1)[\omega]$, with the latter only possibly occurring when $n$ is even.
\end{lem}
\begin{proof}
The reduction mod $2$ of $c_1(M)$ is the second Stiefel-Whitney class $w_2(M)\in H^2(M,\Z_2)$,
which is a topological invariant. Hence it is equal to $w_2(\mathbb{CP}^n)$ which is
$c_1(\mathbb{CP}^n)$ mod $2$, that is $n+1$ mod $2$. On the other hand since $c_1(M)$ and $[\omega]$ both belong to $H^2(M,\mathbb{Z})\cong\mathbb{Z}$, we have $c_1(M)=\lambda [\omega]$ for some $\lambda\in\Z$, and so $\lambda=n+1+2s$ for some $s\in\Z$. From Lemma \ref{lemma2} we get
$$\chi(M,\mathcal{O})=\int_M e^{\frac{n+1+2s}{2}\omega}
\left(\frac{\omega/2}{\sinh (\omega/2)}\right)^{n+1}=\int_M e^{s\omega}\left(\frac{\omega}{1-e^{-\omega}}
\right)^{n+1},$$
using the identity
$$\frac{x}{1-e^{-x}}=e^{\frac{x}{2}} \frac{x/2}{\sinh (x/2)}.$$
Since $\int_M\omega^n=1,$ and the integrals over $M$ of all other powers of $\omega$ are zero by definition, this means that $\chi(M,\mathcal{O})$ equals the coefficient of $x^n$ in the power series expansion of
$$ e^{sx}\left(\frac{x}{1-e^{-x}}\right)^{n+1}.$$
Following \cite{hk} we give two different ways of calculating this coefficient. The first method uses residues, and more precisely the fact that if we define a holomorphic function $F$ by
$$F(z)=e^{sz}\left(\frac{z}{1-e^{-z}}\right)^{n+1},$$
then Cauchy's integral formula shows that the coefficient that we are interested in equals the contour integral
$$\frac{1}{2\pi \sqrt{-1}}\oint \frac{F(z)}{z^{n+1}}dz=\frac{1}{2\pi \sqrt{-1}}\oint \frac{e^{sz}}{(1-e^{-z})^{n+1}}dz,$$
where the countour is a small circle around the origin, with counterclockwise orientation. Since the power series expansion of $1-e^{-z}$ at $z=0$ starts with $z$, this function is a local biholomorphism near the origin, so we can change variable $y=1-e^{-z}$ near $0$ and rewrite our contour integral as
$$\frac{1}{2\pi \sqrt{-1}}\oint \frac{1}{(1-y)^{(s+1)}y^{n+1}}dz,$$
where the contour is again a small circle around the origin. By the Residue theorem this integral equals the residue of the function $\frac{1}{(1-y)^{(s+1)}y^{n+1}}$ at $0$, which is the coefficient of $y^n$ in the Taylor expansion of $(1-y)^{-s-1}$ at $0$. Expanding this function, we finally obtain that our desired coefficient equals
$$\binom{n+s}{n}=\frac{(n+s)(n+s-1)\cdots (s+1)}{n!},$$
where we allows $s<0$.

The second way to calculate this coefficient is as follows: by Hirzebruch-Riemann-Roch again, this coefficient equals
\[\begin{split}
\int_{\mathbb{CP}^n} e^{sH}\left(\frac{H}{1-e^{-H}}\right)^{n+1}&=\int_{\mathbb{CP}^n}e^{sH}e^{\frac{n+1}{2}H}\left(\frac{H/2}{\sinh (H/2)}\right)^{n+1}\\
&=\int_{\mathbb{CP}^n}e^{c_1(\mathcal{O}(s))}e^{\frac{c_1(\mathbb{CP}^n)}{2}}\hat{A}(\mathbb{CP}^n)\\
&=\chi(\mathbb{CP}^n,\mathcal{O}(s)),
\end{split}\]
and it is well-known (see e.g. \cite[Example 5.2.5]{Hu}) that $\chi(\mathbb{CP}^n,\mathcal{O}(s))$ equals
$\binom{n+s}{n}$. So, using either of the two methods, we conclude that
$$\chi(M,\mathcal{O})=\binom{n+s}{n}.$$
Since $\chi(M,\mathcal{O})=1$ by Lemma \ref{lemma1}, we get that
$\binom{n+s}{n}=1,$
which can be rewritten as
$$n!=(s+n)\cdots (s+1).$$
If $n$ is odd this implies that $s=0$, while if $n$ is even, $s$ is either $0$ or $-n-1$.
But we saw that $c_1(M)=(n+1+2s)[\omega]$ and so if $n$ is odd we get $c_1(M)=(n+1)[\omega]$,
while if $n$ is even, $c_1(M)$ is either $(n+1)[\omega]$ or $-(n+1)[\omega]$.
\end{proof}

Assume first that $c_1(M)=(n+1)[\omega]$, which implies that $M$ is a Fano manifold (i.e. there is a K\"ahler metric in $c_1(M)$). Then $c_1(K_M)=-c_1(M)=-(n+1) c_1(L)$ and so $K_M=-(n+1)L$, since the map $c_1$ is
an isomorphism. Then Serre duality
gives $H^k(M,L)\cong H^{n-k}(M,K_M-L)$ and $K_M-L=-(n+2)L$ is negative, so
$H^k(M,L)=0$ if $k>0$ by Kodaira vanishing. Hence using  Hirzebruch-Riemann-Roch again we get
\begin{equation*}
\begin{split}
\dim H^0(M,L)&=\chi(M,L)=\int_M e^{c_1(L)+\frac{c_1(M)}{2}}
\left(\frac{\omega/2}{\sinh (\omega/2)}\right)^{n+1}\\
&=\int_M e^{\omega}\left(\frac{\omega}{1-e^{-\omega}}
\right)^{n+1}=n+1,
\end{split}
\end{equation*}
using again the calculation from earlier of the coefficient in the power series expansion.
Then the following lemma, whose proof we postpone, gives that $M$ is biholomorphic to $\mathbb{CP}^n$.
\begin{lem}[Theorem 1.1 in \cite{ko}]\label{kobo} If $M$ is a compact K\"ahler manifold and $L$ is a positive line bundle on $M$ with $\int_M c_1^n(L)=1$
and $\dim H^0(M,L)=n+1$ then $M$ is biholomorphic to $\mathbb{CP}^n$.
\end{lem}

We can then assume that $n$ is even (so $n\geq 2$) and that $c_1(M)=-(n+1)[\omega]$, which says that
$K_M$ is positive. By a theorem due independently to Yau \cite{yau1} and  Aubin \cite{aubin}
 we know that $M$ then admits a unique K\"ahler-Einstein
metric with constant Ricci curvature equal to $-1$, that is a K\"ahler metric
$\omega_{\rm KE}$ such that
\begin{equation}\label{ke}
\Ric(\omega_{\rm KE})=-\omega_{\rm KE}.
\end{equation}
Recall here that the Riemann curvature tensor of a K\"ahler metric $\omega=\sqrt{-1}g_{i\ov{j}}dz^i\wedge d\ov{z}^j$ in local holomorphic coordinates has components given by
$$R_{i\ov{j}k\ov{\ell}}=-\frac{\de^2g_{k\ov{\ell}}}{\de z^i\de \ov{z}^j}+g^{p\ov{q}}\frac{\de g_{k\ov{q}}}{\de z^i}\frac{\de g_{p\ov{\ell}}}{\de \ov{z}^j},$$
the Ricci curvature tensor is its trace
$$R_{i\ov{j}}=g^{k\ov{\ell}}R_{i\ov{j}k\ov{\ell}}=-\frac{\de^2\log\det(g_{k\ov{\ell}})}{\de z^i\de \ov{z}^j},$$
and the Ricci form is defined by
$$\Ric(\omega)=\sqrt{-1}R_{i\ov{j}}dz^i\wedge d\ov{z}^j,$$
so that the K\"ahler-Einstein condition \eqref{ke} is equivalent to
$$R_{i\ov{j}}=-g_{i\ov{j}}.$$
With this in mind, we have the following:
\begin{lem}
If $(M,\omega)$ is a K\"ahler-Einstein manifold of complex dimension $n\geq 2$, so that
$\Ric(\omega)=\lambda\omega$ for some $\lambda\in \R$, then we have
\begin{equation}\label{chern}
\left(\frac{2(n+1)}{n}c_2(M)-c_1^2(M)\right)\cdot [\omega]^{n-2}\geq 0,
\end{equation}
with equality iff $\omega$ has constant holomorphic sectional curvature.
\end{lem}
\begin{proof}
The tensor $$R^0_{i\ov{j}k\ov{\ell}}=R_{i\ov{j}k\ov{\ell}}-\frac{\lambda}{n+1}(g_{i\ov{j}} g_{k\ov{\ell}}
 +g_{i\ov{\ell}}g_{k\ov{j}})$$
vanishes iff $\omega$ has constant holomorphic sectional curvature (see e.g. \cite[Proposition IX.7.6]{kn}). Its tensorial norm square is easily computed as
\[\begin{split}
|\textrm{Rm}^0|^2&=g^{i\ov{q}}g^{p\ov{j}}g^{k\ov{s}}g^{r\ov{\ell}} R^0_{i\ov{j}k\ov{\ell}} R^0_{p\ov{q}r\ov{s}}\\
&=|\textrm{Rm}|^2+\frac{\lambda^2}{(n+1)^2}g^{i\ov{q}}g^{p\ov{j}}g^{k\ov{s}}g^{r\ov{\ell}}(g_{i\ov{j}} g_{k\ov{\ell}}+g_{i\ov{\ell}}g_{k\ov{j}})
(g_{p\ov{q}} g_{r\ov{s}}+g_{p\ov{s}}g_{r\ov{q}})\\
&-\frac{2\lambda}{n+1}g^{i\ov{q}}g^{p\ov{j}}g^{k\ov{s}}g^{r\ov{\ell}}(g_{i\ov{j}} g_{k\ov{\ell}}+g_{i\ov{\ell}}g_{k\ov{j}})R_{p\ov{q}r\ov{s}}\\
&=|\textrm{Rm}|^2+\frac{\lambda^2}{(n+1)^2}(2n^2+2n)-\frac{4\lambda}{n+1}R,
\end{split}\]
where $R$ denotes the scalar curvature.
The assumption $R_{i\ov{j}}=\lambda g_{i\ov{j}}$ gives $R=\lambda n$ and $|\Ric|^2=\lambda^2n$.
Then
$$|\textrm{Rm}^0|^2=|\textrm{Rm}|^2-\frac{2\lambda^2n}{n+1}.$$
On the other hand if $\Omega_i^j=\sqrt{-1}R^j_{i k\ov{\ell}}dz^k\wedge d\ov{z}^\ell$
denote the curvature forms, then Chern-Weil theory says that
$$\frac{1}{2\pi}\Ric(\omega)=\frac{1}{2\pi}\sum_i \Omega^i_i=\frac{\sqrt{-1}}{2\pi}R_{k\ov{\ell}}dz^k\wedge d\overline{z}^\ell,$$
is a closed form that represents $c_1(M)$ in $H^2(M,\mathbb{R})$, while the form
$$\frac{1}{4\pi^2}\tr(\Omega\wedge\Omega)=\frac{1}{4\pi^2}\sum_{k,i} \Omega_i^k\wedge\Omega_k^i
=\frac{(\sqrt{-1})^2}{4\pi^2}\sum_{k,i}R^k_{ip\ov{q}} R^i_{kr\ov{s}}dz^p\wedge d\ov{z}^q\wedge dz^r\wedge d\ov{z}^s,$$
represents $c_1^2(M)-2c_2(M)$. Since \eqref{chern} is an integral inequality, we can ignore torsion in integral cohomology, and so we can use Chern-Weil forms to prove \eqref{chern}.
Given a point $p\in M$ we choose local holomorphic coordinates so that $p$ we have $g_{i\ov{j}}=\delta_{ij}$, and so also
$$\omega^{n}=n! (\sqrt{-1})^n dz^1\wedge d\overline{z}^1\wedge\cdots\wedge dz^n\wedge d\overline{z}^n,$$
\[\begin{split}\omega^{n-2}=&(n-2)! (\sqrt{-1})^{n-2} \sum_{i<j}dz^1\wedge d\overline{z}^1\wedge\cdots\wedge \widehat{dz^i\wedge d\overline{z}^i}\wedge\cdots\\
&\cdots\wedge  \widehat{dz^j\wedge d\overline{z}^j}\wedge\cdots\wedge dz^n\wedge d\overline{z}^n,\end{split}\]
and it follows that at $p$ we have
\begin{equation*}
\begin{split}
n(n-1)\tr(\Omega\wedge\Omega)\wedge\omega^{n-2}&=\sum_{k,i}\sum_{p\neq r}(R^k_{ip\ov{p}} R^i_{kr\ov{r}}-
R^k_{ip\ov{r}} R^i_{kr\ov{p}})\omega^n\\
&=\sum_{k,i,p,r}(R^k_{ip\ov{p}} R^i_{kr\ov{r}}-
R^k_{ip\ov{r}} R^i_{kr\ov{p}})\omega^n\\
&=(|\Ric|^2-|\textrm{Rm}|^2)\omega^n=(\lambda^2 n- |\textrm{Rm}|^2)\omega^n.
\end{split}
\end{equation*}
Hence this holds at all points, and so
\begin{equation*}
\begin{split}
|\textrm{Rm}^0|^2\frac{\omega^n}{n(n-1)}&=-\tr(\Omega\wedge\Omega)\wedge\omega^{n-2}
+\lambda^2\left(\frac{1}{n-1}-\frac{2}{(n+1)(n-1)}\right)\\
&=-\tr(\Omega\wedge\Omega)\wedge\omega^{n-2}+\frac{\lambda^2}{n+1}.
\end{split}
\end{equation*}
Now notice that
$$\frac{1}{4\pi^2}\int_M\lambda^2\omega^n=\frac{1}{4\pi^2}\int_M (\lambda\omega)^2\wedge\omega^{n-2}=c_1^2(M)\cdot [\omega]^{n-2},$$
and so
$$\frac{1}{n(n-1)4\pi^2}\int_M |\textrm{Rm}^0|^2\omega^n=\left(2c_2(M)-\left(1-\frac{1}{n+1}\right)
c_1^2(M)\right)\cdot[\omega]^{n-2},$$
which implies what we want.
\end{proof}
We claim that equality in \eqref{chern} does in fact hold in our case. This will finish the proof of Theorem \ref{main1}, since
then $M$ would have constant negative holomorphic sectional curvature,
and since it is also simply connected it would be biholomorphic to the unit ball in $\C^n$ (see e.g. \cite[Theorem IX.7.9]{kn}), which is impossible.

We already know that $c_1^2(M)=(n+1)^2[\omega^2]$. To compute $c_2(M)$ we notice that
by definition $p_1(M)=p_1(TM)=-c_2(TM\otimes\C)$. But $TM\otimes\C\cong TM\oplus \overline{TM}$
and the Chern classes satisfy $c_k(\overline{TM})=(-1)^k c_k(TM)$, so
\begin{equation}\label{pont2}
\begin{split}
p_1(M)&=-c_2(TM\oplus\overline{TM})=-c_2(TM)-c_2(\overline{TM})-c_1(TM)\cdot c_1(\overline{TM})\\
&=-2c_2(M)+c_1^2(M).
\end{split}
\end{equation}
Putting this together with \eqref{pont} we get
$$2 c_2(M)=(n+1)^2[\omega^2]-(n+1)[\omega^2]=n(n+1)[\omega^2],$$
and thus equality holds in \eqref{chern}. This completes the proof of Theorem \ref{main1}.
\end{proof}

\begin{proof}[Proof of Theorem \ref{main2}] Let us denote by $\tau(M)$ the signature of $M$, which is the difference between the number of positive and negative eigenvalues for the intersection form $$H_2(M,\mathbb{Z})\times H_2(M,\mathbb{Z})\to\mathbb{Z}.$$
The signature is a topological invariant (up to sign), and so $$\tau(M)=\pm\tau(\mathbb{CP}^2)=\pm 1.$$
Hirzebruch's
Signature Theorem \cite[p.235]{Hu} gives
$$\tau(M)=\frac{1}{3}\int_M p_1(M).$$ But from \eqref{pont2} we get
$$\frac{1}{3}\int_M (c_1^2(M)-2c_2(M))=\pm 1,$$
and Chern-Gauss-Bonnet's Theorem \cite[p.235]{Hu} gives
$$\int_M c_2(M)=\chi(M)=\chi(\mathbb{CP}^2)=3,$$
and so $$\int_M c_1^2(M)= 3(2\pm 1)>0.$$
A theorem of Kodaira \cite{koda} then says that $M$ is projective. As before we see that $\chi(M,\mathcal{O})=1$
and then Riemann-Roch (see \cite[p.233]{Hu}) gives
$$\chi(M,\mathcal{O})=\frac{K_M^2+\chi(M)}{12}=\frac{K_M^2+3}{12},$$
which gives $\int_M c_1^2(M)=K_M^2=9$ (so in fact $\tau(M)=1$). Let $\omega$ be as before, then $c_1(M)=\lambda[\omega]$
for some $\lambda\in\Z$. Then we have that $\lambda=\pm 3$, and these are exactly the same cases as in Theorem \ref{main1}. If $\lambda=3$, we need to check that $\dim H^0(M,L)=3$. But we have $K_M=-3L$ and $K_M\cdot L=-3$ so Riemann Roch \cite[p.233]{Hu} gives
$$\chi(M,L)=1+\frac{L^2-K_M\cdot L}{2}=3.$$
Serre duality and Kodaira vanishing give
$$H^1(M,L)\cong H^1(M,K_M-L)=0,$$
because $K_M-L=-4L$ is negative, and also
$$H^2(M,L)\cong H^0(M,K_M-L)=0.$$
So $\chi(M,L)=\dim H^0(M,L)=3$.
Then the proof continues as in Theorem \ref{main1}.
\end{proof}
\begin{proof}[Proof of Lemma \ref{kobo}]
Let $(\varphi_1,\dots,\varphi_{n+1})$ be a basis of $H^0(M,L)$ and let $D_j=\{\varphi_j=0\}$ be
the corresponding divisors (they are nonempty, because otherwise $L$ would be trivial, and so
it would have $\dim H^0(M,L)=1$). Define $V_n=M$ and $$V_{n-k}=D_1\cap\dots\cap D_k$$
for $1\leq k\leq n$.
\begin{lem}\label{restr}
For each $0\leq r\leq n$ we have that
\begin{itemize}
\item[(1)] $V_{n-r}$ is irreducible, of dimension $n-r$ and Poincar\'e dual to $c_1^r(L)$\\
\item[(2)] The sequence
$$0\to\spn(\varphi_1,\dots,\varphi_r)\to H^0(M,L)\to H^0(V_{n-r},L)$$
is exact, where the last map is given by restriction.
\end{itemize}
\end{lem}
\begin{proof}
The proof is by induction on $r$, the case $r=0$ being obvious. Assuming that $(1)$ and $(2)$ hold for $r-1$, we see that $V_{n-r+1}$ is irreducible and that $\varphi_r$ is not identically zero on it. Hence
$V_{n-r}=\{x\in V_{n-r+1}\ |\ \varphi_r(x)=0\}$ is an effective divisor on $V_{n-r+1}$ and so it can be expressed as a sum of irreducible subvarieties of dimension $n-r$. Since $c_1^{r-1}(L)$ is dual to
$V_{n-r+1}$ and $c_1(L)$ is dual to $D_r$ we see that $c_1^r(L)$ is dual to $V_{n-r}$. If $V_{n-r}$ were reducible, then $V_{n-r}=V'+V''$ and so
\begin{equation*}
\begin{split}
1&=\int_M c_1^n(L)=\int_M c_1^r(L)\cdot c_1^{n-r}(L)=\int_{V_{n-r}} c_1^{n-r}(L)\\
&=\int_{V'} c_1^{n-r}(L)+\int_{V''} c_1^{n-r}(L).
\end{split}
\end{equation*}
But since $L$ is positive, the last two term are both positive integers, and this is a contradiction. Thus $(1)$ is proved. As for $(2)$, the restriction exact sequence
$$0\to \Op{V_{n-r+1}}\to \Op{V_{n-r+1}}(L)\to \Op{V_{n-r}}(L)\to 0,$$
gives
$$0\to H^0(V_{n-r+1},\mathcal{O})\to H^0(V_{n-r+1},L)\to H^0(V_{n-r},L),$$
where the first map is given by multiplication by $\varphi_r$.
This means that the kernel of the restriction map $H^0(V_{n-r+1},L)\to H^0(V_{n-r},L)$
is spanned by $\varphi_r$. This together with the statement in $(2)$ for $r-1$ proves $(2)$ for $r$.
\end{proof}
Now we apply Lemma \ref{restr} with $r=n$ and see that $V_0$ is a single point and that $\varphi_{n+1}$
does not vanish there. So given any point of $M$ there is a section of $L$ that does not vanish there (i.e.
$L$ is base-point-free). Then we can define a holomorphic map $f:M\to\mathbb{CP}^n$
by sending $x$ to $\{\varphi\in H^0(M,L)\ |\ \varphi(x)=0\}$. This is a hyperplane in $H^0(M,L)\cong\C^{n+1}$ and so gives a point in $\mathbb{CP}^n$. If $y\in\mathbb{CP}^n$ corresponds to a hyperplane,
which is spanned by some sections $(\varphi_1,\dots,\varphi_n)$, then $f(x)=y$ iff $\varphi_1(x)=\dots=
\varphi_n(x)=0$. Again Lemma \ref{restr} with $r=n$ says that $x=V_0$ exists and is unique, and so $f$ is
a bijection.
\end{proof}

\section{Closing remarks}\label{sectclose}

As a partial generalization of Theorems \ref{main1} and \ref{main2}, Libgober-Wood \cite{LW} proved that a compact K\"ahler manifold of complex dimension $n\leq 6$ which is homotopy equivalent to
$\mathbb{CP}^n$ must be biholomorphic to it.

A natural question is whether the K\"ahler hypothesis is really necessary in Theorem \ref{main1}, and so one can ask whether a compact complex manifold diffeomorphic to $\mathbb{CP}^n$ must be biholomorphic to it. This is a well-known open problem (see e.g. \cite{LW}), and it is known that if this is true when $n=3$ then there is no complex manifold diffeomorphic to $S^6$ (another famous open problem, see e.g. \cite{Le}):
\begin{prop}
If there exists a compact complex manifold $M$ diffeomorphic to $S^6$, then there exists a compact complex manifold $\ti{M}$ diffeomorphic to $\mathbb{CP}^3$ but not
biholomorphic to it.
\end{prop}
This well-known fact was remarked already in \cite[p.223]{Hi2}.
\begin{proof}
Let $M$ be a compact complex manifold diffeomorphic to $S^6$, and let $\ti{M}$ be its blowup at one point $p\in M.$ This is a compact complex manifold which is diffeomorphic to the connected sum $S^6\sharp \overline{\mathbb{CP}^3}$, see e.g. \cite[Proposition 2.5.8]{Hu}. This is of course diffeomorphic to $\overline{\mathbb{CP}^3}$,
and so also to $\mathbb{CP}^3$ (in fact, it is even oriented-diffeomorphic to $\mathbb{CP}^3$, since this manifold has the explicit orientation-reversing diffeomorphism $[z_0:\dots:z_3]\mapsto [\overline{z_0}:\dots:\overline{z_3}]$).
 So $\ti{M}$ is diffeomorphic to $\mathbb{CP}^3$, and if $\ti{M}$ was biholomorphic to $\mathbb{CP}^3$ we would have
$$\int_{\ti{M}}c_1(\ti{M})^3=\int_{\mathbb{CP}^3}c_1(\mathbb{CP}^3)^3=64.$$
But if we let $\pi:\ti{M}\to M$ be the blowup map and $E=\pi^{-1}(p)$ be its exceptional divisor (which is biholomorphic to $\mathbb{CP}^2$), then we have (see \cite[Proposition 2.5.5]{Hu})
$$c_1(\ti{M})=\pi^*c_1(M)-2[E],$$
where $[E]$ denotes the Poincar\'e dual of $E$. Since $b_2(M)=0$ we have $c_1(M)=0$, and so
$$\int_{\ti{M}}c_1(\ti{M})^3=-8\int_M [E]^3=-8\int_{E} [E]^2=-8\int_{\mathbb{CP}^2}c_1(\mathcal{O}(-1))^2=-8,$$
since $[E]|_E$ equals the first Chern class of the tautological bundle $\mathcal{O}(-1)$ over $\mathbb{CP}^2$ (see  \cite[Corollary 2.5.6]{Hu}). Therefore $\ti{M}$ is not biholomorphic to $\mathbb{CP}^3$, as claimed.
\end{proof}

\end{document}